\documentclass[10pt]{article}

\usepackage{pb-diagram}
\usepackage{times}
\usepackage{amsthm}
\usepackage{amsmath}
\usepackage{amssymb}
\usepackage{mathrsfs}
\usepackage{comment}
\usepackage[english]{babel}
\usepackage{amsfonts}
\usepackage[pdftex]{graphicx}
\usepackage{anysize}
\usepackage{setspace}
\usepackage{textcomp}

\newtheorem{Theorem}{Theorem}[section]
\newtheorem{Remark}[Theorem]{Remark}
\newtheorem{Lemma}[Theorem]{Lemma}

\newtheorem{Proposition}[Theorem]{Proposition}
\newtheorem{Definition}[Theorem]{Definition}

\def\R{\mathbb{R}}
\def\N{\mathbb{N}}

\def\d{\partial}
\def\dd{\partial^2}

\def\g{\gamma}
\def\G{\Gamma}
\def\l{\lambda}
\def\e{\varepsilon}
\def\a{\alpha}

\def\w{\omega}
\def\t{\tau}
\def\lf{\left(}
\def\rg{\right)}
\def\p{\cdot}

\def\bt{\bar{t}}

\def\cv{\mathcal{V}}
\def\vn{\vec{n}}

\newcommand{\Rj}[1]{\mathbb{R}^{ #1}}
\newcommand{\f}[2]{\frac{#1}{#2}}

\newcommand{\I}[2]{\int_{#1}^{#2}}

\begin{document}

\title{\bf NON EXISTENCE OF TYPE II SINGULARITIES FOR EMBEDDED AND UNKNOTTED SPACE CURVES}

\date{\today}

\author{KAREN CORRALES \footnote{\textbf{Key Words:} Space curves, Curve Shortening Flow, Total Curvature, Isoperimetric Inequality, Minimal Surfaces.}\\
}
\date{\small}

\maketitle \vspace{-1cm}


\begin{abstract}
In this paper we prove that a certain class of embedded unknotted curves in $\mathbb{R}^3$ evolving  under curve shortening flow do not form singularities Type II before collapsing to a point. Our proof uses tools of the minimal surface theory to study a suitable isoperimetric ratio.
\end{abstract}

\section{Introduction}\label{intro}
Let $\g: S^1\!\times\![0,\w)\to \Rj{3}$ be a smooth family of embedded space curves. We say that $\g$ evolves by the curve shortening flow if
\begin{equation}
\label{curveshorteningflow}
\tag{CSF}
\f{\d}{\d t}\g(\p,t)=k(\p,t) N(\p,t),
\end{equation}
where $k(\p,t)$ and $N(\p,t)$ are the curvature and the normal vector of $\g(\cdot,t)$, respectively.

This flow was proposed in 1956 by Mullins to model the motion of idealized grain boundaries. In the context of geometric measure theory Brakke studied weak solutions of the mean curvature flow in 1978, of which the curve shortening flow is the 1-dimensional case. In 1986, Grayson, Gage and Hamilton renewed interest in the curve shortening flow arising from work on planar curves.

A central topic within the subject is to understand the singularities that the curves may develop during the evolution. For instance, the formation of singularities is fully understood for planar curves that are smooth and embedded. In \cite{HeatequationG&H} Gage and Hamilton proved that if the initial curve is closed, convex and embedded in $\R^2$ then the solution to (\ref{curveshorteningflow}) keeps convex during the evolution and when $t$ tends to the first singular time it shrinks to a point. The characterization of singularities of closed embedded curves in the plane was completed by Grayson in \cite{HeatequationG}, where he proved that if the initial curve is closed, embedded and possibly not convex in $\R^2$ then the solution to (\ref{curveshorteningflow}) becomes a convex curve before the first singular time, thus it converges to a point as well. A simplified proof of these theorems was proven by Huisken in \cite{Distanceprincipio} where he used   extrinsic and intrinsic distances to define certain isoperimetric ratios.

In contrast, to study the behavior of curves in $\R^3$ (codimension 2) that evolves by curvature is more difficult than planar curves. For example, they may not remain embedded and inflection points may develop \cite{Ramps}. Consequently, fewer results are known in that context and usually it is necessary to have some preserved quantities.

The main goal of this paper is to study the formation of singularities of solutions to \eqref{curveshorteningflow} when $\gamma(\cdot,t)$ is an unknotted and embedded space curve for every $t\in[0,\w)$. More precisely, we will focus on discarding the formation of Type II singularities under appropriate assumptions (see  \ref{definitionsingularities} for a definition of Type II singularities).

Our main result is the following theorem

\medskip
\noindent {\bfseries Theorem A}
{\it \hspace{.1cm}Suppose that $\g(\p,t)$ satisfies the equation (\ref{curveshorteningflow}) for $t\in[0,\omega)$ and it becomes singular as $t\to\omega$. Assume additionally that the total curvature of $\g(\p,0)$ is less than $4\pi$ and that $\g(\p,t)$ does not converge to a point as $t\to\omega$. Let $X(\p,t)$ be the regular in space-time minimal surface enclosed by $\g(\p,t)$. If the Gaussian curvature $K(\p,t)$ of $X(\p,t)$ is uniformly bounded and $\g(\p,t)$ remains embedded for every $t\in[0,\w)$ with finite singular time $\w$ and does not shrink to a point, then $\g(\p,t)$ does not develop Type II singularities.}

\medskip
This theorem complements the result proved in \cite{distancespace}, which establishes that if the initial curve satisfies that its total curvature is less than $4\pi$, and the singularity formed in the curve shortening flow is of Type I, then the curve shrinks to a round point when approaching the maximum existing singular time.

The proof of Theorem A relies on the fact that curves with total curvature less than $4\pi$ enclose a unique minimal surface (see \cite{Uniquenessminimalsurface}). Hence, for every $\g(\p,t)$ we may consider the surface $X(\p,t)$ with boundary $\g(\p,t)$ and define an isoperimetric ratio in the spirit of \cite{Distanceprincipio}. Specifically, in \cite{Distanceprincipio} G.Huisken studied the formation of singularities of planar curves by comparing the intrinsic distance in the curve (arc-length between the points) with the extrinsic distance on the plane (the usual Euclidean distance). The new approach in this work will be consider the extrinsic distance in $\R^3$ as the length of a piecewise geodesic on the minimal surface $X(\p,t)$ which join two points on $\g(\p,t)$. The main result follows by a contradiction argument that compares the bounds on the isoperimetric ratio of the rescaled sequence of $\g(\p,t)$ and the one of its limit.

\medskip
The organization of the paper is as follows: In section \ref{preliminares}, we will collect basic facts and notation about the curve shortening flow, total curvature, rescaled solution and singularities. We will also state the results that will be used freely in this work. Section \ref{surface} will be devoted to analyze the existence, uniqueness and evolution of the minimal surface $X(\p,t)$ with boundary $\g(\p,t)$. In particular, we will show properties that ensure that, under appropriate rescaling, this family of surfaces converges to a limit minimal surface. Section \ref{sectionradio} contains the definition of the isoperimetric ratio $\mathcal{G}(a,b,t)$ for $\g(\p,t)$, which is key in the proof of the main theorem, and estimates for its spatial and time variations. Finally, using the results in Section \ref{surface} and \ref{sectionradio}, we will prove Theorem A in the last section.

\medskip
{\bf Acknowledgements:} This work is included in the author's doctoral dissertation for the Universidad de Chile. The author wishes to thank her advisor, Mariel S\'aez, for many informative discussions and suggestions in the author's thesis work.

The author is supported by Beca Doctorado Nacional, CONICYT.
\section{Preliminaries}\label{preliminares}
In order to fix notation, the Frenet matrix for a space curve $\g$ with arc-length parameter $s$ will be written as:
\begin{eqnarray}
\label{frenet}
\frac{\d}{\d s} \left(\begin{array}{c}
T\\
N\\
B
\end{array}\right)= \left(\begin{array}{ccc}
0&k&0\\
-k&0&\tau\\
0&-\tau&0
\end{array}\right) \left(\begin{array}{c}
T\\
N\\
B
\end{array}\right);
\end{eqnarray}
where $T,N$ and $B$ are the tangent, normal and binormal vectors, respectively. The quantities $k$ and $\t$ are the curvature and torsion of $\g$.

\begin{Remark}{\rm
Note that the curvature of a space curve is always defined as a non-negative quantity.
}
\end{Remark}

The curve shortening flow is defined as follows
\begin{Definition}
A space curve $\g$ evolves by the {\bfseries \itshape curve shortening flow} if it satisfies the equation (\ref{curveshorteningflow}), where $\g: S^1\times[0,\w)\to\R^3$ is a one-parameter smooth family of curves.
\end{Definition}

\begin{Remark}{\rm
It is important to remark that the equation \eqref{curveshorteningflow} is not invariant under reparametrizations of the family $\g(\p,t)$. However, it is possible to prove that every tangential expression can be obtained by a reparametrization. Therefore, the following equation is geometrically equivalent to \eqref{curveshorteningflow} and it is invariant under parametrizations.
\begin{equation}
\label{CSFgeneral}
\f{\d \g}{\d t}(p,t)\p N(p,t)=k(p,t).
\end{equation}}
\end{Remark}

Short time existence of solutions to this differential equation follows from a general theorem proved in \cite{HeatequationG&H}.
\begin{Theorem}
Let $\g_0$ be a smooth, immersed and closed curve in $\R^3$. There exists $\e>0$ such that solutions $\g:S^1\times[0,\e)\to\R^3$ to \eqref{curveshorteningflow} exist. Furthermore, these solutions are smooth.
\end{Theorem}

The following theorem proved in \cite{Ramps} states that, regardless of the behavior of $\tau$, bounded curvature $k$ implies long-time existence.
\begin{Theorem}
\label{existencialong}
If the curvature of $\g$ is uniformly bounded on the time interval $[0,\a)$, there exists an $\e>0$ such that $\g(\p,t)$ exists and is smooth on the extended time
interval $[0,\a+\e).$
\end{Theorem}

Consequently, if the curvature becomes unbounded then there exists a maximal time of evolution $\w$. In this situation we say that the evolution of $\g$ by its curvature forms a singularity at that time. Singularity formation can be classified according to the following definitions.

\begin{Definition}
\label{definitionsingularities}
Given a space curve $\g(\p,t)$ that evolves by curve shortening flow which forma a singularity at $\w$, we say:
\begin{itemize}
\item The singularity formation is {\bfseries \itshape Type I} if $\,\lim_{t\to \w}(\sup k^2(\cdot,t))(\w-t)$ is bounded.
\item The singularity formation is {\bfseries \itshape Type II} if $\,\lim_{t\to \w}(\sup k^2(\cdot,t))(\w-t)$ is unbounded.
\end{itemize}
\end{Definition}

To study the formation of these singularities we define two types of sequences:
\begin{Definition}
\label{blowupsequence}
For $\g(\p,t)$ defined as in Definition \ref{definitionsingularities}, we say:
\begin{itemize}
\item $\{(p_n,t_n)\}\in S^1\!\times\![0,\w)$ is a {\bfseries \itshape blow-up sequence} if
\begin{equation*}
 \lim_{n\to \infty}t_n=\w \quad \mbox{ and }\quad \,\lim_{n \to \infty} k^2(p_n,t_n)=\infty;
\end{equation*}
\item $\{(p_n,t_n)\}\in S^1\!\times\![0,\w)$ is an {\bfseries \itshape essential blow-up sequence} if it is a blow-up sequence and there exists $\rho \in \R^{+}$, independent of $n$, such that
\begin{equation*}
 \rho \,(\,\sup\,k^2(\cdot,t))\leq k^2(p_n,t_n) \mbox{ when } t \leq t_n.
\end{equation*}
\end{itemize}
\end{Definition}

A standard way to analyze singularities is under rescalings in space and in time to obtain a limit curve with bounded curvature. A rescaled solution of a curve $\g$ that evolves by the curve shortening flow is defined as follows
\begin{Definition}
\label{definicionrescalamiento}
A rescaled solution $\g_n$ of $\g$ along a blow-up sequence $\{(p_n,t_n)\} \in S^1\!\times\![0,\w)$ is a curve \linebreak $\g_n: S^1\times[-\l^2_{n}t_n,\l^2_{n}(\w-t_n))\to \Rj{3}$ defined by:
\begin{eqnarray*}
\g_n(\cdot,\overline{t})=\l_n(O_n\g(\cdot,t)+B_n); \hspace{0.5 cm} \overline{t}=\l^2_n(t-t_n),
\end{eqnarray*}
where $\l_n \in \R^{+}, O_n \in SO(3), B_n \in \Rj{3}$ are chosen so that $\g_n$ is a solution to \eqref{curveshorteningflow} and:
\begin{itemize}
  \item $\g_n(p_n,0)=0 \in \Rj{3}$;
  \item the unit tangent vector $T_n(p_n,0)=(1,0,0)$;
  \item $k_n \cdot N_n(p_n,0)=(0,1,0)$.
\end{itemize}
\end{Definition}

\begin{Remark}
\label{parametrizationg}
{\rm Since the limit curve of the rescaled solution may not be closed, it is convenient to assume that each $\g_n$ is defined on the real line as a periodic map. Thus, we assume that $\g(\p,t):\lf-\f{\pi}{2},\f{\pi}{2}\right]\to \R^3.$ }
\end{Remark}

In \cite[Th.6.1]{Singularities} it was proved that if $\{(p_n,t_n)\}$ is an essential blow-up sequence, then the formation of singularities is a planar phenomenon in the following sense:

\begin{Theorem}
\label{fenomenoplanar}
If $\{(p_n,t_n)\}$ is an essential blow-up sequence then
\begin{eqnarray*}
\lim_{n\to\infty}\f{\tau}{k}(p_n,t_n)=0.
\end{eqnarray*}
\end{Theorem}

Thus, analogously to the planar case (according to \cite{Singularities}) and following Definition \ref{definitionsingularities} we have

\begin{Proposition}
\label{TypeII}
If $\g(\p,t)$ evolves by curve shortening flow and it forms a singularity at time $\w$ then its limit of rescaled solutions satisfies:
\begin{itemize}
\item If $\g$ forms a Type I singularity, then $\g$ is asymptotic to a planar solution which is homothetically shrinking, i.e. it is a contracting self-similar solution. These planar solutions were studied and classified by Abresch-Langer \cite{AbreschLanger} (see Fig.\ref{ALcurves}).
\item If $\g$ forms a Type II singularity, then there exists an essential blow-up sequence $\{(p_n,t_n)\}$ such that a rescaling of $\g$ converges along a subsequence of $\{(p_n,t_n)\}$ to a convex eternal solution $\g_\infty$.
\item Convex eternal solutions are characterized by a translating graph that satisfies an ordinary differential equation and are known as the {\it Grim Reaper} (Definition \ref{DefinicionGrimReaper}).
\end{itemize}
\end{Proposition}

\begin{Definition} (\cite{HeatequationG})
\label{DefinicionGrimReaper}
The Grim Reaper is the planar curve defined by $$\g_\infty(x,t)=(x,-\ln(\cos(x))+t),\quad \mbox{for}\, x\in\lf-\frac{\pi}{2},\frac{\pi}{2}\rg,$$
thus it moves upwards with constant speed in time (see Fig.\ref{ginf}).
\end{Definition}

\begin{Remark}
{\rm The Grim Reaper is a convex curve with bounded curvature that evolves by the curve shortening flow following \eqref{CSFgeneral}.}
\end{Remark}

\begin{figure}[h]
\centering
\includegraphics[width=6cm,height=2.7cm]{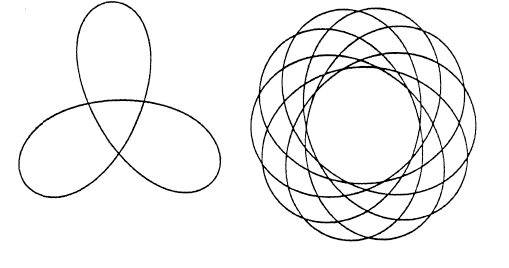} 
\caption{Abresch-Langer Curves}
\label{ALcurves}
\end{figure}

\begin{figure}[h]
\centering
\includegraphics[width=3.7cm,height=4.5cm]{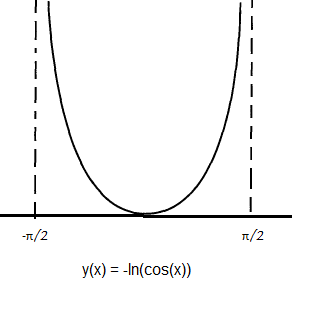} 
\caption{Grim Reaper}
\label{ginf}
\end{figure}

The total curvature is a relevant geometric quantity for embedded space curves that will be considered in this work to guarantee the existence and uniqueness of the minimal surface $X(\p,t)$ for every $t\in[0,\w)$ (Section \ref{surface}). In \cite[Th.5.1]{Singularities} it was shown that if a curve evolves by its curvature, then the total curvature has a monotone behavior in the evolution. More precisely,
\begin{Theorem}
\label{totalcurvaturedecreasing}
Let $\g$ be a solution to \eqref{curveshorteningflow}. Then we have
\begin{equation}
\f{\d}{\d t}\I{\g}{}|k|ds\leq -\I{\g}{}\tau^2|k|ds<0.
\end{equation}
\end{Theorem}

Additionally, following \cite[Sec.5-7,Th.3]{Docarmo} we have a direct relation between the total curvature and geometric properties of a space curve.

\begin{Theorem}[Fenchel's Theorem]
\label{total curvature}
If $\g$ is an embedded curve then the total curvature $\int_{\g} |k| \geq 2\pi$ and equality holds if and only if the curve is a planar convex curve.
\end{Theorem}

In this work, we will consider a sort of embedded curves: unknotted curves.

\begin{Definition} \cite{Unknotted} An embedded curve $\g: S^1 \to \Rj{3}$ is {\bfseries \itshape unknotted} if there is an orientation-preserving homeomorphism of $\Rj{3}$ onto itself which maps $\g$ onto a planar circle in $\Rj{3}$, i.e. onto $S^1;$ otherwise, $\g$ is {\bfseries \itshape knotted} or is a knot.
\end{Definition}

\begin{figure}[h]\label{knot}
\centering
\includegraphics[width=6cm,height=3.5cm]{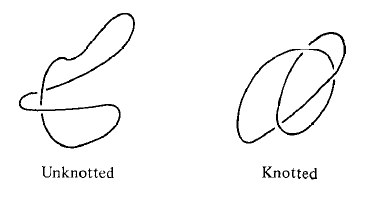} 
\caption{}
\end{figure}

Moreover, from \cite{Totalcurvatureknot} we have the following
\begin{Lemma}
\label{unknottedesmenorque4pi}
If the curve $\g$ is knotted then its total curvature is greater or equal than $4\pi$. Equivalently, if the total curvature of $\g$ is less than $4\pi$ then $\g$ is an unknotted curve.
\end{Lemma}

Thus, using Theorem \ref{totalcurvaturedecreasing}, it is easy to prove
\begin{Proposition}
\label{siempreesunknotted}
Suppose that $\g(\p,t)$ satisfies the equation (\ref{curveshorteningflow}) and the total curvature of $\g(\p,0)$ is less than $4\pi$. If $\g(\p,t)$ remains embedded for all $t\in[0,\w)$ then $\g(\p,t)$ will be unknotted for every $t\in[0,\w)$.
\end{Proposition}

\begin{Remark}
{\rm In this work, we will study space curves $\g(\p,t)$ that satisfies the equation (\ref{curveshorteningflow}) and remains embedded for all $t\in[0,\w)$. Moreover, we will assume that $\g(\p,0)$ has total curvature less than $4\pi$.}
\end{Remark}

\section{Evolution of the minimal surface enclosed by $\g(\p,t)$}
\label{surface}
This section is devoted to study the minimal surface $X(\cdot,t)$ enclosed by $\g(\cdot,t)$. We start by recalling that Plateau's problem (originally proved in \cite{Existenceminimal}) asserts that for each embedded curve in $\mathbb{R}^3$ is possible to find an enclosed minimal surface. Moreover, results in \cite{Uniquenessminimalsurface} imply that if the boundary is unknotted then the enclosed minimal surface
is unique, disc-type and area minimizing.

\begin{Remark}
\label{Atwelldefined}
{\rm Note that if we assume that $\g_0$ has total curvature less than $4\pi$ and $\g(\p,t)$ remains embedded for every $t\in[0,\w)$, Proposition \ref{siempreesunknotted} implies that $\g(\p,t)$ will be unknotted for every $t\in[0,\w)$. Therefore, from \cite{Uniquenessminimalsurface} we will have that there exists a unique associated minimal surface $X(\p,t)$.}
\end{Remark}

In addition, since we will use convergence of a rescaled solution to a limit planar curve, we will need results on convergence of sequences of minimal surfaces.

\begin{Theorem}(\cite[Th.25]{Convergence})
\label{convergence}
Suppose that $\{M_n\}_{n\in\N}\subset \R^3$ is a sequence of orientable, embedded, simply connected minimal surfaces such that $\I{M_n}{}|K_n|dA_n<\lambda<\infty$.
Then (after passing to a subsequence) $\{M_n\}_{n\in\N}$ converges (at least in $C^{1,\a}$) in $\R^3$ to a limit minimal surface $M$.
\end{Theorem}

Under the hypotheses on $\g(\p,t)$ in this paper, we have

\begin{Proposition}
\label{boundedK}
Let $X(\p,t)$ be a minimal surface where its boundary $\g(\p,t)$ has total curvature less than $4\pi$, then $\I{X_t}{}|K(\p,t)|dA<2\pi$ for all $t\in[0,\w)$, where $K(\p,t)$ is the Gaussian curvature of the surface $X(\p,t)$.
\end{Proposition}

\begin{proof}
Since $X(\p,t)$ is a minimal surface, its Gaussian curvature $K(\p,t)$ will be non-positive. Using the Gauss-Bonnet Theorem ($\chi=1$) we get
\begin{equation*}
\I{X(\p,t)}{}|K(\p,t)|dA=-\I{X(\p,t)}{}K(\p,t) dA=\I{\g(\p,t)}{}k_g(\p,t) ds-2\pi\leq\I{\g(\p,t)}{}|k_g(\p,t)|ds-2\pi,
\end{equation*}
where $k_g(\p,t)$ is the geodesic curvature of $\g(\p,t).$

On the other hand, if $k(\p,t)$ is the curvature of $\g(\p,t)$ then $$k^2(\p,t)=k_N^2(\p,t)+k_g^2(\p,t),$$ where $k_N(\p,t)$ is the normal curvature and $k_g(\p,t)$ is the geodesic curvature. Therefore, $|k_g(\p,t)|\leq|k(\p,t)|.$

Since $\g(\p,t)$ has total curvature less than $4\pi$, then
$$\I{X_t}{}|K(\p,t)|dA\leq\I{\g(\p,t)}{}|k_g(\p,t)|ds-2\pi<2\pi.$$
\end{proof}

\begin{Remark}
{\rm It is important to remark that in this paper we will assume that the minimal surface $X(\p,t)$ is continuous and regular up to the boundary and in time.}
\end{Remark}

\begin{Theorem}
\label{convergeofsurface}
Let $\{X_n\}_{n\in\N}$ be a sequence of the associated minimal surfaces of the rescaled solution $\g_n(\p,0)$ defined previously. If the Gaussian curvature of $X(\p,t)$ is uniformly bounded for every $t\in[0,\w)$, then $\{X_n\}_{n\in\N}$ converges smoothly to a planar surface.
\end{Theorem}

\begin{proof}
Firstly, if $K_n$ denotes the Gaussian curvature of $X_n(\p,0)$, it is easy to compute that
\begin{equation}
\label{gaussiancurvature}
|K_n(\p,0)|=\l_n^{-2}|K(\p,t_n)|.
\end{equation}
Then, following Proposition \ref{boundedK} we get
\begin{equation}
\label{boundedtgc}
\I{X_n(\p,0)}{}|K_n(\p,0)|dA_n=\I{X(\p,t_n)}{}\l_n^{-2}|K(\p,t_n)|\l_n^2\p dA_{t_n}=\I{X(\p,t_n)}{}\left|K\lf\p,t_n\rg\right|dA_{t_n}<2\pi.
\end{equation}
Thus, $\{X_n\}_{n\in\N}$ is a sequence of minimal surfaces with total Gaussian curvature less than $2\pi$. Following \cite{Uniquenessminimalsurface} we obtain that the sequence $\{X_n\}_{n\in\N}$ are minimal surfaces of disk-type, i.e. embedded and simply connected. Therefore, using Theorem \ref{convergence} we get that the sequence $\{X_n\}_{n\in\N}$ converges smoothly to a minimal surface.

On the other hand, from equation (\ref{gaussiancurvature}) we can conclude that $K_n(\p,0)\to 0$ a.e. when $n\to\infty$. Since,
$|A_n(\p,0)|^2=-2K_n(\p,0)$, we obtain that the limit is a planar minimal surface.
\end{proof}

\begin{Remark}
\label{convergenceuptheboundary}
{\rm In this work we will assume that $X(\p,t)$ is at least $C^{1,\a}$ in time. Therefore, we may suppose that the convergence in Theorem \ref{convergeofsurface} is $C^{1,\a}$ up to the boundary and by regularity theory, we may obtain that this convergence is smooth up to the boundary.}
\end{Remark}

In this paper we will denote by $\f{d}{dt}X(u,v,t)$ the total derivate of $X(u,v,t)$ with respect $t$, by $X_u$ and $X_v$ the partial derivatives respect to $u$ and $v$, respectively.

The next lemma shows that if we assume that the minimal surface $X(\p,t)$ is regular in time then its evolution is given by
\begin{equation}
\label{evolX}
\f{d}{dt}X(\p,t)=\mathcal{V}(\p,t)\vec{n}(\p,t),
\end{equation}
where $\vec{n}(\p,t)$ is the unit normal vector to the surface $X(\p,t)$ and $\cv(\p,t)$ is a real smooth function.

\begin{Lemma}
\label{evolutionX}
Let $X:B\times[0,\w)\to\R^3$ be a parametrization of the minimal surface enclosed by $\g(\p,t)$ regular in time. If the evolution of $X$ is given by
$$\f{d}{dt}X(u,v,t)=\mathcal{V}(u,v,t)\vec{n}(u,v,t)+X_T(u,v,t),$$
where $\mathcal{V}$, $\vec{n}$ and $X_T$ denote a real function, the inward unit normal vector and a vector in the tangent plane to the surface $X(\p,t)$, respectively. Then, there exists a reparametrization $\tilde{X}$ of $X$ such that
\begin{equation*}
\f{d}{dt}\tilde{X}(u,v,t)=\tilde{\mathcal{V}}(u,v,t)\tilde{\vec{n}}(u,v,t).
\end{equation*}
\end{Lemma}

\begin{proof}
Let $\varphi(u,v,t)$ be a smooth family of diffeomorphisms of $B$ such that
\begin{equation*}
\f{d}{d t}\varphi(u,v,t)=-\lf\f{d}{dt}X(u,v,t)\p X_T(u,v,t)\rg\quad\mbox{and}\quad \varphi|_{S^1}(\p,0)={\it id}|_{S^1}.
\end{equation*}

We consider the following reparametrization of $X$
$$\tilde{X}(u,v,t)=X(\varphi(u,v,t),t).$$
Thus, the chain rule implies
\begin{eqnarray*}
\f{d}{dt}\tilde{X}(u,v,t)&=&\mathcal{V}(\varphi(u,v),t)\vec{n}(\varphi(u,v),t).
\end{eqnarray*}
\end{proof}

Now, we will prove that the function $\cv(\p,t)$ is related with the Gaussian curvature of $X(\p,t)$ in the following way

\begin{Lemma}
\label{cotaV}
Under the same hypothesis of Lemma \ref{evolutionX}, the normal component of the evolution of $X(\p,t)$ satisfies
\begin{equation*}
\lf \f{d}{dt}X(\p,t)\p\vec{n}(\p,t)\rg^2=\mathcal{V}^2(\p,t)\leq |K(\p,t)|.
\end{equation*}
\end{Lemma}

\begin{proof}
We will prove that $\mathcal{V}$ satisfies an elliptic equation.

Firstly, since $\f{d}{d t}X=\cv \vn$ then
\begin{eqnarray*}
\f{d}{d t}g_{ij}&=&\f{d}{dt}\d_i X\p\d_j X+\d_i X\p\f{d}{dt}\d_j X\\
&=&\d_i(\cv\vn)\p\d_j X+\d_iX\p\d_j(\cv\vn)\\
&=&2\cv h_{ij}.
\end{eqnarray*}
This implies
\begin{equation}
\label{evolutiongij}
\f{d}{dt} g^{ij}=-2\cv h^{ij}.
\end{equation}

Thus,
\begin{equation}
\label{deltaV}
\Delta \cv=\Delta\lf\f{d}{dt} X\p\vn\rg=\Delta\lf\f{d}{dt} X\rg\p\vn+2\nabla\lf\f{d}{dt} X\rg\p\nabla \vn+\f{d}{dt} X\p\Delta \vn.
\end{equation}

Moreover, note that
\begin{eqnarray*}
\f{d}{dt} \lf\Delta X\rg&=&\f{d}{dt} \lf\f{1}{\sqrt{g}}\d_i(\sqrt{g}g^{ij}\d_j X)\rg\\
&=&\f{1}{\sqrt{g}}\d_i(\sqrt{g}\f{d}{dt} g^{ij}\d_j X)+\Delta\lf\f{d}{dt} X\rg.
\end{eqnarray*}
Since $X$ is a minimal surface, then
\begin{equation}
\label{1}
\begin{split}
\Delta\lf\f{d}{dt} X\rg&=-\f{1}{\sqrt{g}}\d_i\lf\sqrt{g}\f{d}{dt} g^{ij}\d_j X\rg\\
&\underset{\eqref{evolutiongij}}{=}\f{2\cv}{\sqrt{g}}\d_i(\sqrt{g}h^{ij}\d_j X)+2h^{ij}\d_i \cv\d_j X.
\end{split}
\end{equation}

On the other hand, using \eqref{evolutiongij} we get
\begin{equation}
\label{2}
\begin{split}
2\nabla\lf\f{d}{dt} X\rg\p\nabla \vn&=2g^{ij}\d_i\lf\f{d}{dt} X\rg\p \d_j \vn\\
&=2\cv g^{ij}(\d_i\vn\p \d_j\vn),
\end{split}
\end{equation}
and
\begin{equation}
\label{3}
\begin{split}
\f{d}{dt} X\p\Delta \vn&=\f{d}{dt} X\p \lf\f{1}{\sqrt{g}}\d_i(\sqrt{g}g^{ij}\d_j\vn)\rg\\
&=-\cv g^{ij}(\d_i\vn\p\d_j\vn).
\end{split}
\end{equation}
Thus, replacing \eqref{1}, \eqref{2} and \eqref{3} in \eqref{deltaV}, we obtain
\begin{eqnarray*}
\Delta \cv&=&\f{2\cv}{\sqrt{g}}\d_i(\sqrt{g}h^{ij}\d_j X)\p\vn+2h^{ij}\d_i\cv(\d_j X\p\vn)+\cv g^{ij}(\d_i\vn\p \d_j\vn)\\
&=&\f{2\cv}{\sqrt{g}}\d_i(\sqrt{g}h^{ij}\d_j X)\p\vn+\cv g^{ij}(\d_i\vn\p \d_j\vn)\\
&=&2\cv|A|^2+\cv g^{ij}(\d_i\vn\p \d_j\vn).
\end{eqnarray*}
Since this quantity is invariant under coordinates changes, we have
\begin{equation}
\Delta\cv=3\cv|A|^2=6\cv|K|.
\end{equation}
Therefore, if we consider the elliptic operator $\mathcal{L}(u)=\Delta u-6|K|u$ and $k_N(\p,t)$ denotes the normal curvature of $\g(\p,t)$ then we have
\begin{equation*}
\left\{ \begin{array}{cl}
\mathcal{L}(\mathcal{V})=0 &\mbox{in}\, B \\
\mathcal{V}=k_N & \mbox{on}\, \d B.
\end{array} \right.
\end{equation*}

Thus, using the maximum principle we get
\begin{equation*}
\mathcal{V}^2=\lf\f{d}{d t}X\p\vec{n}\rg^2\leq \lf\f{d}{d t}\g\p\vec{n}\rg^2=\lf k\,N\p\vec{n}\rg^2=(k_N)^2\leq |K(\p,t)|.
\end{equation*}
The last inequality comes from the Euler's Theorem (on principle curvatures) and the minimality of $X(\p,t).$
\end{proof}

\section{Isoperimetric Ratio}
\label{sectionradio}
In this section we will define and study the isoperimetric ratio, that it will be the key to prove Theorem A in the next section.

Recall that we always assume that $\g(\p,t)$ is an embedded curve such that $\g(\p,0)$ has total curvature less than $4\pi$ and evolves by curve shortening flow. Following Section \ref{surface} we have that there exists a unique disc-type minimal surface $X(\p,t)$ with boundary $\g(\p,t)$.


We start by defining a piecewise geodesic minimizing and the isoperimetric ratio. Next, we will study its spatial variation to obtain its behavior through the evolution by curvature of $\g(\p,t)$

\begin{Definition}
\label{def-hyp}
For fixed $t<\w$, consider $\g(\p,t):\left(\f{-\pi}{2},\f{\pi}{2}\right]\to\R^3$ a space curve evolving by curve shortening flow such that the total curvature of $\g(\p,0)$ is less than $4\pi$ and $\g(\p,t)$ remains embedded for every $t\in[0,\w)$ with finite singular time $\w$. Let $X(\p,t)$ be the regular in space-time minimal surface enclosed by $\g(\p,t)$ with Gaussian curvature $K(\p,t)$ uniformly bounded in time. Given two different points $\g(a,t)$ and $\g(b,t)$ on $\g(\p,t)$, we define the extrinsic distance ${\it d}_{ab}^{\,t}$ the minimum of the lengths of curves that $\g(a,t)$ and $\g(b,t)$, this distance is realized by a piecewise geodesic on $X(\p,t)$.
\end{Definition}

We will fix a piecewise geodesic that realizes the distance ${\it d}_{ab}^{\,t}$ and denote by
$$\G_{ab}(x,t)=X(u(x,t),v(x,t),t),\quad (x,t)\in [0,1]\times[0,\w),$$
such that $$\G_{ab}(0,t)=\g(b,t),\quad \G_{ab}(1,t)=\g(a,t);$$
while, the arc-length of $\g(\p,t)$ between $\g(a,t)$ and $\g(b,t)$ as ${\it l}_{ab}^{\,t}$ (intrinsic distance).

Following \cite{Distanceprincipio}, if $L_t$ is the total length of $\g(\p,t)$, the intrinsic distance function is only smoothly defined for $0\leq {\it l}_{ab}^{\,t}< \f{L_t}{2}$ with conjugate points when ${\it l}_{ab}^{\,t}=\f{L_t}{2}$. Thus, we consider a smooth function $\psi$ defined by
\begin{equation}
\psi(a,b,t):=\f{L_t}{\pi}\sin\lf\f{\pi {\it l}_{ab}^{\,t}}{L_t}\rg,
\end{equation}

and the isoperimetric ratio $\mathcal{G}$ defined by
\begin{equation}
\label{definitionradio}
\mathcal{G}(a,b,t):=\f{\psi(a,b,t)}{{\it d}_{ab}^{\,t}}.
\end{equation}

\begin{Remark}
{\rm It is easy to see that for every $t\in[0,\w)$, the isoperimetric ratio $\mathcal{G}$ has its global minimum when $a=b$. Thus, we will study the value maximum of $\mathcal{G}.$}

Therefore, we may assume w.l.o.g that for a fixed $t$ the maximum is attained at $a\ne b$. Fixing a point we define the arc-length parameter and assume that $s(a)<s(b)$.
\end{Remark}

For the proof of Theorem A in the next section, it is necessary to understand the behavior in time of $\mathcal{G}$. For this, we need to study its spatial variation.

Let $a_\e$ and $b_\e$ be the variations of $a$ and $b$, respectively, such that
\begin{equation*}
\left.\f{d}{d\e}\right|_{\e=0}s(a_\e)=\mathcal{A},\quad \left.\f{d}{d\e}\right|_{\e=0}s(b_\e)=\mathcal{B}\quad\mbox{and}\quad \left.\f{d^2}{d\e^2}\right|_{\e=0}s(a_\e)=\left.\f{d^2}{d\e^2}\right|_{\e=0}s(b_\e)=0.
\end{equation*}

For fixed $t\in[0,\w)$, we define $\G_\e(\p,t):[0,1]\to X(\p,t)$ as a variation of $\G_{ab}(\p,t)$ such that
$$\G_\e(0,t)=\g(b_\e,t),\quad \G_\e(1,t)=\g(a_\e,t)\quad\mbox{and}\quad \G_0(x,t)=\G_{ab}(x,t).$$

If ${\it d}^{\,t}_\e$ denotes the length of $\G_\e(\p,t)$ and ${\it l}_\e^{\,t}$ denotes the length of $\g(\p,t)$ between $a_\e$ and $b_\e$, then we define
$$\psi(\e,t):=\f{L_{t}}{\pi}\sin\lf\f{\pi {\it l}_{\e}^{\,t}}{L_{t}}\rg\quad\mbox{and}\quad \mathcal{G}(a_\e,b_\e,t)=\f{\psi(\e,t)}{{\it d}_\e^{\,t}}.$$

In this section we will denote by $T_\e(\p,t)$ and $T_\g(\p,t)$ to the unit tangent vector to $\G_\e(\p,t)$ and the unit tangent vector to $\g(\p,t),$ respectively.

To analyze the first spatial variation of $\mathcal{G}$, we will start by computing the first spatial variation of ${\it d}_{ab}^{\,t}$ and $\psi(a,b,t).$

\begin{Proposition}
\label{d_e}
Under the hypotheses of Definition \ref{def-hyp}, for a fixed $t$ assume that the maximum of $\mathcal{G}(\cdot,\cdot,t)$ is attained at $(a,b)$. Then the first spatial variation of ${\it d}_{ab}^{\,t}$ is given by
\begin{equation*}
\left.\f{d}{d \e}\right|_{\e=0}{\it d}_\e^{\,t}=\mathcal{A}T_0(1,t)\p T_\g(a,t)-\mathcal{B}T_0(0,t)\p T_\g(b,t).
\end{equation*}
\end{Proposition}

\begin{proof}
We have that $\G_\e(x,t)=X(u(x,t,\e),v(x,t,\e),t)$ and ${\it d}_\e^{\,t}=\I{\G_\e}{}ds_\e(t)$. Thus, we compute
\begin{equation}
\label{ded}
\begin{split}
\f{d}{d \e}{\it d}_\e^{\,t}&=\f{d}{d \e}\I{0}{1}\left\|\f{\d}{\d x}\G_\e(\p,t)\right\| dx\\
&=\I{0}{1}T_\e(\p,t)\p\f{\d}{\d x}\f{d}{d \e}\G_\e(\p,t)dx\\
&=T_\e(\p,t)\p\left.\f{d}{d \e}\G_\e(\p,t)\right|_0^1-\I{0}{1}\f{\d}{\d x}T_\e(\p,t)\p\f{d}{d \e}\G_\e(\p,t)dx.
\end{split}
\end{equation}

Since $\G_\e(\p,t)$ is a piecewise geodesic, it satisfies the following equations
\begin{equation}
\label{tangencial}
\f{\d^2}{\d x^2}\G_\e(\p,t)\p X_u(\p,t)=\f{\d^2}{\d x^2}\G_\e(\p,t)\p X_v(\p,t)=0\quad\mbox{and}\quad \f{\d^2}{\d x^2}\G_\e(\p,t)\p \vn(\p,t)=II_{t}(u_x,v_x),
\end{equation}
where $II_{t}$ is the second fundamental form associated to $X(\p,t)$.

On the other hand, we know that
\begin{eqnarray*}
\f{\d}{\d x} T_\e(\p,t)&=&\f{1}{\left\|\f{\d}{\d x}\G_\e(\p,t)\right\|}\f{\d^2}{\d x^2}\G_\e(\p,t)+\f{1}{\left\|\f{\d}{\d x}\G_\e(\p,t)\right\|^2}\lf\f{\d^2}{\d x^2}\G_\e(\p,t)\p \f{\d}{\d x}\G_\e(\p,t)\rg T_\e(\p,t)\\
&\underset{\eqref{tangencial}}{=}&\f{1}{\left\|\f{\d}{\d x}\G_\e(\p,t)\right\|}\f{\d^2}{\d x^2}\G_\e(\p,t),
\end{eqnarray*}

and this implies
$$\f{\d}{\d x}T_\e(\p,t)\p\f{d}{d \e}\G_\e(\p,t)=\f{\d}{\d x}T_\e(\p,t)\p\lf u_\e X_u(\p,t)+v_\e X_v(\p,t)\rg=0.$$

Thus, following \eqref{ded} we obtain
\begin{equation}
\label{firstvariationd}
\f{d}{d \e}{\it d}_\e^{\,t}=\lf\f{d}{d\e}s(a_\e)\rg T_\e(1,t)\p T_\g(a_\e,t)-\lf\f{d}{d\e}s(b_\e)\rg T_\e(0,t)\p T_\g(b_\e,t).
\end{equation}

Therefore, $$\left.\f{d}{d \e}\right|_{\e=0}{\it d}_\e^{\,t}=\mathcal{A}T_0(1,t)\p T_\g(a,t)-\mathcal{B}T_0(0,t)\p T_\g(b,t).$$
\end{proof}

\begin{Proposition}
\label{psi_e}
Under the hypotheses of Proposition \ref{d_e}, the first spatial variation of $\psi(a,b,t)$ is given by
$$\left.\f{d}{d \e}\right|_{\e=0}\psi(\e,t)=\lf\mathcal{B}-\mathcal{A}\rg\cos\lf\f{\pi {\it l}_0^{\,t}}{L_{t}}\rg.$$
\end{Proposition}

\begin{proof}
Throughout this proof we assume that $t$ is fixed. We compute
\begin{equation}
\label{firstvariationpsi}
\begin{split}
\f{d}{d\e}\psi(\e,t)&=\cos\lf\f{\pi {\it l}_\e^{\,t}}{L_{t}}\rg\p\f{d}{d\e}{\it l}_\e^{\,t}\\
&=\cos\lf\f{\pi {\it l}_\e^{\,t}}{L_{t}}\rg\lf\f{d}{d\e}s(b_\e)-\f{d}{d\e}s(a_\e)\rg.
\end{split}
\end{equation}

Therefore, $$\left.\f{d}{d \e}\right|_{\e=0}\psi(\e,t)=\cos\lf\f{\pi {\it l}_0^{\,t}}{L_{t}}\rg\lf\mathcal{B}-\mathcal{A}\rg.$$
\end{proof}

Using Proposition \ref{d_e} and \ref{psi_e} we obtain the first spatial variation of $\mathcal{G}(\cdot,\cdot,t)$ at its maximum

\begin{Lemma}
\label{firstrelation}
Under the hypotheses of Lemma \ref{d_e}, the tangent vectors $T_\g(\p,t)$ and $T_0(\p,t)$ satisfy
\begin{equation*}
T_0(1,t)\p T_\g(a,t)=T_0(0,t)\p T_\g(b,t)=-\cos\lf\f{\pi {\it l}_0^{\,t}}{L_{t}}\rg\p\f{{\it d}_0^{\,t}}{\psi(0,t)}.
\end{equation*}
\end{Lemma}

\begin{proof}
The first spatial variation of the isoperimetric ratio $\mathcal{G}(a,b,t)$ is given by
\begin{equation}
\f{d}{d\e}\mathcal{G}(a_\e,b_\e,t)=\f{1}{{\it d}_\e^{\,t}}\f{d}{d\e}\psi(\e,t)-\f{\psi(\e,t)}{\lf{\it d}_\e^{\,t}\rg^2}\f{d}{d\e}{\it d}_\e^{\,t}.
\end{equation}

Since the maximum of $\mathcal{G}(a_\e,b_\e,t)$ is attained at $\e=0$, from Proposition \ref{d_e} and \ref{psi_e} we get
\begin{equation*}
0=\left.\f{d}{d\e}\right|_{\e=0}\mathcal{G}(a_\e,b_\e,t)=\f{\cos\lf\f{\pi {\it l}_0^{\,t}}{L_{t}}\rg}{{\it d}_0^{\,t}}\lf\mathcal{B}-\mathcal{A}\rg-\f{\psi(0,t)}{\lf{\it d}_0^{\,t}\rg^2}\lf\mathcal{A}T_0(1,t)\p T_\g(a,t)-\mathcal{B}T_0(0,t)\p T_\g(b,t)\rg.
\end{equation*}

Thus, if we take $\mathcal{B}=\mathcal{A}=1$, we obtain
\begin{equation*}
T_0(1,t)\p T_\g(a,t)=T_0(0,t)\p T_\g(b,t).
\end{equation*}
On the other hand, if we set $\mathcal{B}=1, \,\mathcal{A}=0$, we obtain
\begin{equation*}
T_0(0,t)\p T_\g(b,t)=-\cos\lf\f{\pi {\it l}_0^{\,t}}{L_{t}}\rg\p\f{{\it d}_0^{\,t}}{\psi(0,t)}.
\end{equation*}
\end{proof}

For simplicity we will denote $$\a^{t}:=\lf\f{\pi {\it l}_0^{\,t}}{L_{t}}\rg\quad \mbox{and}\quad \a_\e^{t}:=\lf\f{\pi {\it l}_\e^{\,t}}{L_{t}}\rg. $$

Next, to obtain the second spatial variation of $\mathcal{G}(a,b,t)$ we need to compute the second spatial variation of ${\it d}_{ab}^{\,t}$ and $\psi(a,b,t)$. Following \eqref{firstvariationd} and \eqref{firstvariationpsi} we get
\begin{equation}
\label{secondvariationd}
\begin{split}
\left.\f{d^2}{d\e^2}\right|_{\e=0}{\it d}_\e^{\,t}=&\mathcal{A}\lf\left.\f{d}{d\e}\right|_{\e=0}T_\e(1,t)\rg\p T_\g(a,t)-\mathcal{B}\lf\left.\f{d}{d\e}\right|_{\e=0}T_\e(0,t)\rg\p T_\g(b,t)\\
&+\mathcal{A}^2k_\g(a,t) T_0(1,t)\p N_\g(a,t)-\mathcal{B}^2k_\g(b,t) T_0(0,t)\p N_\g(b,t),
\end{split}
\end{equation}

and,
\begin{equation}
\label{secondvariationpsi}
\left.\f{d^2}{d\e^2}\right|_{\e=0}\psi(\e,t)=-\sin\lf\a^{t}\rg\p\f{\pi}{L_{t}}\p\lf\mathcal{B}-\mathcal{A}\rg^2.
\end{equation}

Using these computations we obtain the second spatial variation of $\mathcal{G}(\cdot,\cdot,t)$ at its maximum

\begin{Lemma}
\label{secondrelation}
Under the hypotheses of Lemma \ref{d_e}, the curvature $k_\g(\p,t)$ of $\g(\p,t)$, the tangent vector $T_0(\p,t)$ of $\G_{0}(\p,t)$ and the normal vector $N_\g(\p,t)$ of $\g(\p,t)$ satisfy
$$k_\g(a,t) T_0(1,t)\p N_\g(a,t)-k_\g(b,t) T_0(0,t)\p N_\g(b,t)\geq -4{\it d}_0^{\,t}\lf\f{\pi}{L_{t}}\rg^2.$$
\end{Lemma}

\begin{proof}
The second spatial variation of the isoperimetric ratio $\mathcal{G}(a,b,t)$ is given by
$$\f{d^2}{d\e^2}\mathcal{G}(a_\e,b_\e,t)=\f{1}{{\it d}_\e^{\,t}}\f{d^2}{d\e^2}\psi(\e,t)-\f{2}{\lf{\it d}_\e^{\,t}\rg^2}\f{d}{d\e}{\it d}_\e^{\,t}\f{d}{d\e}\psi(\e,t)-\f{\psi(\e,t)}{\lf{\it d}_\e^{\,t}\rg^2}\f{d^2}{d\e^2}{\it d}_\e^{\,t}+\f{2\psi(\e,t)}{\lf{\it d}_\e^{\,t}\rg^3}\lf\f{d}{d\e}{\it d}_\e^{\,t}\rg^2.$$

Since the maximum of $\mathcal{G}(a_\e,b_\e,t)$ is attained at $\e=0$, using Proposition \ref{d_e} and \ref{psi_e}, and equations \eqref{secondvariationd} and \eqref{secondvariationpsi}, we have
\begin{eqnarray*}
0&\geq&\left.\f{d^2}{d\e^2}\right|_{\e=0}\mathcal{G}(a_\e,b_\e,t)\\
&=&-\left.\f{1}{\lf{\it d}_0^{\,t}\rg^2}\right[\f{\pi\p{\it d}_0^{\,t}}{L_{t}}\p\sin\lf\a^t\rg\lf\mathcal{B}-\mathcal{A}\rg^2\\
&&\hspace{1.3cm}+2\cos\lf\a^t\rg\p\lf\mathcal{B}-\mathcal{A}\rg\lf\mathcal{A}T_0(1,t)\p T_\g(a,t)-\mathcal{B}T_0(0,t)\p T_\g(b,t)\rg\\
&&\hspace{1.3cm}+\psi(0,t)\mathcal{A}^2k_\g(a,t) T_0(1,t)\p N_\g(a,t)-\psi(0,t)\mathcal{B}^2k_\g(b,t) T_0(0,t)\p N_\g(b,t)\\
&&\hspace{1.3cm}+\psi(0,t)\mathcal{A}\lf\left.\f{d}{d\e}\right|_{\e=0}T_\e(1,t)\rg\p T_\g(a,t)-\psi(0,t)\mathcal{B}\lf\left.\f{d}{d\e}\right|_{\e=0}T_\e(0,t)\rg\p T_\g(b,t)\\
&&\hspace{1.3cm}-\left.\f{2\psi(0,t)}{{\it d}_0^{\,t}}\lf\mathcal{A}T_0(1,t)\p T_\g(a,t)-\mathcal{B}T_0(0,t)\p T_\g(b,t)\rg^2\right].
\end{eqnarray*}

Note that Lemma \ref{firstrelation} implies
\begin{equation}
\label{T}
T_0(1,t)\p T_\g(a,t)+T_0(0,t)\p T_\g(b,t)=-2\cos\lf\a\rg\p \f{{\it d}_0^{\,t}}{\psi(0,t)}.
\end{equation}

Thus, using \eqref{T} and we take $\mathcal{A}=1, \,\mathcal{B}=-1$ we obtain
\begin{equation}
\label{a1b-1}
\begin{split}
0&\geq-\left.\f{1}{\lf{\it d}_0^{\,t}\rg^2}\right[\f{4\pi\p{\it d}_0^{\,t}}{L_{t}}\p\sin\lf\a\rg\\
&\hspace{1.3cm}+\psi(0,t)k_\g(a,t) T_0(1,t)\p N_\g(a,t)-\psi(0,t)k_\g(b,t) T_0(0,t)\p N_\g(b,t)\\
&\hspace{1.3cm}+\left.\psi(0,t)\lf\left.\f{d}{d\e}\right|_{\e=0}T_\e(1,t)\rg\p T_\g(a,t)+\psi(0,t)\lf\left.\f{d}{d\e}\right|_{\e=0}T_\e(0,t)\rg\p T_\g(b,t)\right].
\end{split}
\end{equation}

On the other hand, setting $\mathcal{A}=-1, \,\mathcal{B}=1$ we get
\begin{equation}
\label{a-1b1}
\begin{split}
0&\geq-\left.\f{1}{\lf{\it d}_0^{\,t}\rg^2}\right[\f{4\pi\p{\it d}_0^{\,t}}{L_{t}}\p\sin\lf\a\rg\\
&\hspace{1.3cm}+\psi(0,t)k_\g(a,t) T_0(1,t)\p N_\g(a,t)-\psi(0,t)k_\g(b,t) T_0(0,t)\p N_\g(b,t)\\
&\hspace{1.3cm}-\left.\psi(0,t)\lf\left.\f{d}{d\e}\right|_{\e=0}T_\e(1,t)\rg\p T_\g(a,t)-\psi(0,t)\lf\left.\f{d}{d\e}\right|_{\e=0}T_\e(0,t)\rg\p T_\g(b,t)\right].
\end{split}
\end{equation}

Combining the inequalities \eqref{a1b-1} and \eqref{a-1b1} we obtain

\begin{equation*}
k_\g(a,t) T_0(1,t)\p N_\g(a,t)-k_\g(b,t) T_0(0,t)\p N_\g(b,t)\geq -4{\it d}_0^{\,t}\lf\f{\pi}{L_{t}}\rg^2.
\end{equation*}
\end{proof}

\begin{Remark}
{\rm Note that if $\g(a,t)$ and $\g(b,t)$ are conjugate points for fixed time $t$ (i.e. ${\it l}_{ab}^{\,t}=\f{L_t}{2}$) then $\psi(a,b,t)=\f{L_t}{\pi}$. Therefore, using Proposition \ref{psi_e} and Lemma \ref{firstrelation}, we obtain that the variations $\g(a_\e,t)$ and $\g(b_\e,t)$ satisfy
$$\left.\f{d}{d\e}\right|_{\e=0}\psi(\e,t)=0\quad\mbox{and}\quad T_0(1,t)\p T_\g(a,t)=T_0(0,t)\p T_\g(b,t)=0.$$

Moreover, in this case, Lemma \ref{secondrelation} is satisfied as well.}
\end{Remark}

Now, we want to estimate the behavior of the evolution in time of $\mathcal{G}(a,b,t)$. We start by computing the evolution of ${\it d}_{ab}^{\,t}$ and $\psi(a,b,t)$, fixing $a,b \in\left(\f{-\pi}{2},\f{\pi}{2}\right]$.

\begin{Lemma}
\label{evolutiond}
Let $\g(\p,t):\left(\f{-\pi}{2},\f{\pi}{2}\right]\to\R^3$ be a space curve evolving by curve shortening flow such that the total curvature of $\g(\p,0)$ is less than $4\pi$ and $\g(\p,t)$ remains embedded for every $t\in[0,\w)$ with finite singular time $\w$. Let $X(\p,t)$ be the regular in space-time minimal surface enclosed by $\g(\p,t)$ with Gaussian curvature $K(\p,t)$ uniformly bounded in time. Fixing two different points $\g(a,t)$ and $\g(b,t)$ on $\g(\p,t)$, the evolution in time of ${\it d}_0^{\,t}$ is given by
\begin{equation*}
\f{d}{d t}{\it d}_0^{\,t}=k_\g(a,t)T_0(1,t)\p N_\g(a,t)-k_\g(b,t)T_0(0,t)\p N_\g(b,t)-\I{0}{1}\cv(\p,t) k_0^N(\p,t)ds(t).
\end{equation*}
Here $\cv(\p,t)=\f{d}{dt}X(\p,t)\p\vn(\p,t)$ (defined in \eqref{evolX}) and $k_0^N(\p,t)$ is the normal curvature of $\G_0(\p,t).$
\end{Lemma}

\begin{proof}
We know that
\begin{equation}
\label{evold}
\f{d}{d t}{\it d}_0^{\,t}=\I{0}{1}\f{d}{d t}\left\|\f{\d}{\d x}\G_0(\p,t)\right\|dx=\I{0}{1}T_0(\p,t)\p \f{d}{d t}\f{\d}{\d x}\G_0(\p,t) dx.
\end{equation}

A straightforward computation shows that

\begin{equation}
\label{evold2}
\f{d}{d t}\f{\d}{\d x}\G_0(\p,t)=\f{\d}{\d x}\f{d}{d t}\G_0(\p,t).
\end{equation}

Therefore, using that $\G_0(\p,t)$ is a piecewise geodesic (\eqref{tangencial}) and combining \eqref{evold}, \eqref{evold2} and the definition of normal curvature $k_0^N(\p,t)$, we have
\begin{eqnarray*}
\f{d}{d t}{\it d}_0^{\,t}&=&\I{0}{1}T_0(\p,t)\p\f{\d}{\d x}\f{d}{d t}\G_0(\p,t)dx\\
&=&T_0(\p,t)\p \left.\f{d}{d t}\G_0(\p,t)\right|_0^1-\I{0}{1}\f{\d}{\d x}T_0(\p,t)\p \f{d}{d t}\G_0(\p,t)dx\\
&=&k_\g(a,t)T_0(1,t)\p N_\g(a,t)-k_\g(b,t)T_0(0,t)\p N_\g(b,t)-\I{0}{1}\f{1}{\left\|\f{\d}{\d x}\G_0(\p,t)\right\|}\f{\dd}{\d x^2}\G_0(\p,t)\p\f{d}{d t}X(\p,t)\,dx\\
&=&k_\g(a,t)T_0(1,t)\p N_\g(a,t)-k_\g(b,t)T_0(0,t)\p N_\g(b,t)-\I{0}{1}\cv(\p,t)k_0^N(\p,t) ds(t).
\end{eqnarray*}
\end{proof}

\begin{Lemma}
\label{evolutionpsi}
Under the same hypotheses of Lemma \ref{evolutiond}, $\psi(0,t)$ is a decreasing function in time.
\end{Lemma}

\begin{proof}
We compute
\begin{equation}
\label{evolpsi}
\f{d}{dt}\psi(0,t)=\f{1}{\pi}\lf\f{d}{dt}L_t\rg\sin(\a^t)+\cos(\a^t)\p\f{d}{dt}{\it l}_0^{\,t}-\f{{\it l}_0^{\,t}}{L_t}\p\cos(\a^t)\p\f{d}{dt}L_t,
\end{equation}
where $\a^t=\f{\pi {\it l}_0^{\,t}}{L_{t}}.$

Moreover, we have
\begin{equation}
\label{evolL}
\begin{split}
\f{d}{d t}L_t&=\f{d}{d t}\I{\g_t}{}ds(t)\\
&=\I{\f{-\pi}{2}}{\f{\pi}{2}}\f{d}{d t}\left\|\f{\d}{\d x}\g(\p,t)\right\|dx\\
&=\I{\f{-\pi}{2}}{\f{\pi}{2}}\f{\d}{\d x}(k_\g(\p,t)\p N_\g(\p,t))\p T_\g(\p,t)dx\\
&=\I{\g_t}{}\lf\f{\d}{\d s}k_\g(\p,t)N_\g(\p,t)-k_\g^2(\p,t)T_\g(\p,t)+k_\g(\p,t)\t_\g(\p,t) B_\g(\p,t)\rg\p T_\g(\p,t)ds(t)\\
&=-\I{\g_t}{}k_\g^2\,d s(t)\leq 0.
\end{split}
\end{equation}

Consequently, the arc-length between $\g(a,t)$ and $\g(b,t)$ satisfies
\begin{equation}
\label{evoll}
\f{d}{dt}{\it l}_0^{\,t}=-\I{a}{b}k_\g^2\,d s(t)\leq 0.
\end{equation}

Replacing \eqref{evolL} and \eqref{evoll} in \eqref{evolpsi}, we obtain
\begin{equation}
\label{evolpsi2}
\begin{split}
\f{d}{dt}\psi(0,t)&=-\f{1}{\pi}\sin(\a^t)\p\I{\g}{}k_\g^2\,ds(t)-\cos(\a^t)\p\I{a}{b}k_\g^2\,ds(t)+\f{{\it l}_0^{\,t}}{L_t}\p\cos(\a^t)\p\I{\g}{}k_\g^2\,ds(t)\\
&=-\f{1}{\pi}\sin(\a^t)\I{\g}{}k_\g^2\,ds(t)\lf 1-\f{\a^t}{\tan(\a^t)}\rg-\cos(\a^t)\p\I{a}{b}k_\g^2\,ds(t).
\end{split}
\end{equation}

Since the intrinsic distance ${\it l}_0^{\,t}$ is only smoothly defined for $0\leq{\it l}_0^{\,t}<\f{L_t}{2}\,$ then $\a\in[0,\f{\pi}{2})$. Thus, we get that $$\,\sin(\a^t)\geq 0,\quad\cos(\a^t)\geq0\quad \mbox{and} \quad\f{\a^t}{\tan(\a^t)}<1.$$

This finishes the proof.
\end{proof}

\begin{Remark}
{\rm Note that Lemma \ref{evolutiond} and \ref{evolutionpsi} hold for any $a,b\in\left(\f{-\pi}{2},\f{\pi}{2}\right]$. Therefore, if $a,b$ are conjugate points (${\it l}_{ab}^{\,t}=\f{L_t}{2}$) Lemma \ref{evolutiond} and \ref{evolutionpsi} hold as well.}
\end{Remark}

Thus, we conclude
\begin{Theorem}
\label{evolutionradio2}
Under the hypotheses of Theorem A, for a fixed $t$ assume that the maximum of $\mathcal{G}(\cdot,\cdot,t)$ is attained at $(a,b)$. Then the isoperimetric ratio is bounded for every $t\in[0,\w)$.
\end{Theorem}

\begin{proof}
Since the evolution of $\mathcal{G}(a,b,t)$ is given by
\begin{equation*}
\f{d}{d t}\mathcal{G}(a,b,t)=\f{1}{{\it d}_0^{\,t}}\f{d}{dt}\psi(0,t)-\f{\psi(0,t)}{({\it d}_0^{\,t})^2}\f{d}{dt}{\it d}_0^{\,t},
\end{equation*}

then using Lemma \ref{evolutiond} and \ref{evolutionpsi} we get
\begin{equation}
\label{evolG}
\begin{split}
\f{d}{d t}\mathcal{G}(a,b,t)&\leq-\f{\psi(0,t)}{({\it d}_0^{\,t})^2}\lf k_\g(a,t)T_0(1,t)\p N_\g(a,t)-k_\g(b,t)T_0(0,t)\p N_\g(b,t)\rg\\
&+\f{\psi(0,t)}{({\it d}_0^{\,t})^2}\p\I{0}{1}\cv(\p,t) \p k_0^N(\p,t) ds(t).
\end{split}
\end{equation}

Moreover, Lemma \ref{cotaV} implies that
\begin{equation}
\label{v}
\I{0}{1}\mathcal{V}(\p,t)\p k_0^N(\p,t)ds(t)\leq \I{0}{1}|K(\p,t)|ds(t)\leq {\it d}_0^{\,t}\p \lf\sup_{t\in[0,\w)}|K(\p,t)|\rg.
\end{equation}

Since the maximum of $\mathcal{G}$ is attainted at $(a,b)$, then replacing the result of Lemma \ref{secondrelation} and the inequality \eqref{v} in \eqref{evolG}, we obtain
\begin{equation*}
\f{d}{d t}\mathcal{G}(a,b,t)\leq \mathcal{G}(a,b,t)\lf \f{4\pi^2}{L_{t}^2}+\sup_{t\in[0,\w)}|K(\p,t)|\rg.
\end{equation*}

On the other hand, the Gaussian curvature of $X(\p,t)$ is uniformly bounded for every $t\in[0,\w)$ and $L_{t}>0$ ($\g(\p,t)$ does not shrink to a point), then there exists $\mathcal{C}\in\R$ such that
\begin{equation}
\label{evolG2}
\f{d}{d t}\mathcal{G}(a,b,t)\leq \mathcal{C}\p\mathcal{G}(a,b,t).
\end{equation}

Therefore, the Maximum Principle implies
\begin{equation}
\label{boundedG1}
\mathcal{G}(a,b,t)\leq \mathcal{G}(a,b,0)\exp\lf\mathcal{C}(t)\rg.
\end{equation}

In addition, it is easy to see that $\mathcal{G}(a,b,0)$ is bounded using the Mean value Theorem. Therefore, equation \eqref{boundedG1} implies that the isoperimetric ratio $\mathcal{G}$ is bounded for every $t\in[0,\w)$.
\end{proof}

\begin{Remark}
{\rm It is important to recall that the assumption on $|K(\p,t)|$ (uniformly bounded) implies that the minimal surface is at least $C^{1,\a}$ in time.}
\end{Remark}

\begin{Remark}
\label{invariante}
{\rm It is easy to show that the isoperimetric ratio $\mathcal{G}(a,b,t)$ is invariant under the rescaling along a blow-up sequence (Definition \ref{definicionrescalamiento}), i.e.
\begin{equation*}
\lf\mathcal{G}(a,b,\bt)\rg_n=\f{\psi(0,\bt)_n}{\lf{\it d}_{0}^{\,\bt}\rg_n}=\f{\psi(0,t)}{{\it d}_{0}^{\,t}}=\mathcal{G}(a,b,\bt).
\end{equation*}}
\end{Remark}

Thus, we may prove that the rescaled geodesic $\lf\G_{ab}(\p,\bt)\rg_n$ tends to a straight line when $n$ tends to infinity.

\begin{Proposition}
\label{curvaturetendzero}
If the Gaussian curvature $K(\p,t)$ of $X(\p,t)$ is uniformly bounded in time then the curvature of $\lf\G_{ab}(\p,\bt)\rg_n$ tends to zero when $n$ tends to infinity.
\end{Proposition}

\begin{proof}
If $(k_{ab})_n(\p,\bt)$ denotes the curvature of $\lf\G_{ab}(\p,\bt)\rg_n$, then we have
\begin{eqnarray*}
(k_{ab})_n(\p,\bt)&=&\f{k_{ab}(\p,t)}{\l_n},
\end{eqnarray*}
where $k_{ab}(\p,t)$ is the curvature of $\G_{ab}(\p,t)$ and $\l_n=k_\g^2(p_n,t_n)$ is the spatial scale of the rescaling.

On the other hand, as $\G_{ab}(\p,t)$ is a piecewise geodesic then $(k_{ab}(\p,t))^2$ is equal to its normal curvature $((k_{ab}^N)(\p,t))^2.$

Therefore, as $|K(\p,t)|$ is uniformly bounded, we conclude
\begin{equation*}
((k_{ab})_n(\p,\bt))^2=\f{((k_{ab})(\p,t))^2}{\l_n^2}=\f{((k_{ab}^N)(\p,t))^2}{\l_n^2}\leq \f{|K(\p,t)|}{\l_n^2}\underset{n\to\infty}{\longrightarrow} 0.
\end{equation*}
\end{proof}
\section{Main Theorem}
\label{maintheorem}
In this section we conclude the proof of Theorem A using a contradiction argument. More precisely, we will assume that $\g(\cdot,t)$ forms a Type II singularity and show that the isoperimetric ratio $\mathcal{G}$ cannot remain bounded. That contradicts the estimates obtained in Section \ref{sectionradio}.

We start by establishing notation: Let $\g(\p,t)$ be a space curve that evolves by curve shortening flow such that it remains embedded for every $t\in[0,\w)$ and $\g(\p,0)$ has total curvature less than $4\pi.$

Suppose that the evolution of $\g(\p,t)$ forms a singularity at time $\w$. Thus, we consider the sequence of rescaled solutions $\{\g_n(\p,\bt)\}_{n\in\N}$ of $\g(\p,t)\,$ (Definition \ref{definicionrescalamiento}). If we assume that $\g(\p,t)$ forms a Type II singularity at time $\w$, then, using Proposition \ref{TypeII}, there exists an essential blow-up sequence $\{(p_n, t_n)\}$ such that a limit of rescalings along this sequence converges in $C^1$ uniformly on compact subsets of $\R\times[-\infty,\infty]$ to the Grim Reaper.

\begin{Remark}
{\rm Since we know that $\R$ is homeomorphic to any open interval, we will consider that $\g_n$ converges uniformly on compact sets of
$(-\f{\pi}{2},\f{\pi}{2})\times[-\infty,\infty]$ to $\g_\infty.$ Moreover, following Remark \ref{parametrizationg}, we may assume that the parametrization of $\g_n$ is given by:
\begin{equation}
\g_n:\left(-\f{\pi}{2},\f{\pi}{2}\right]\times [-\l^2_{n}t_n,\l^2_{n}(\w-t_n)) \longrightarrow \R^3.
\end{equation}}
\end{Remark}

In this section, we will consider the following compact subset of $(-\f{\pi}{2},\f{\pi}{2})\times\R$,
$$\mathcal{K}_m=\left[-\a(m),\a(m)\right]\times\{0\},$$
where $\a(m):=\lf\f{\pi}{2}-\f{1}{m}\rg.$

For the endpoints of this compact set, we consider the isoperimetric ratio defined in Section \ref{sectionradio} for the space curve $\g$ and for the sequence of rescaled solutions $\{\g_n\}_{n\in\N}$, in the following way:

We consider the piecewise geodesic $\G_m(\p,t_n)$ that minimizes the distance between $\g(-\a(m),t_n)$ and $\g(\a(m),t_n).$ If ${\it d}^{\,t_n}_m$ denotes the length of $\G_m(\p,t_n)$ and ${\it l}_m^{\,t_n}$ denotes the length of $\g(\p,t_n)$ between $-\a(m)$ and $\a(m)$, then we define
$$\psi(m,t_n):=\f{L_{t_n}}{\pi}\sin\lf\f{\pi {\it l}_{m}^{\,t_n}}{L_{t_n}}\rg\quad\mbox{and}\quad\mathcal{G}(m,t_n)=\f{\psi(m,t_n)}{{\it d}_m^{\,t_n}}.$$

Analogously in the rescaled case, we consider the rescaled piecewise geodesic $(\G_m(\p,0))_n$ that minimizes the distance between $\g_n(-\a(m),0)$ and $\g_n(\a(m),0).$ If $({\it d}_m^{\,0})_n$ denotes the length of $(\G_m(\p,0))_n$ and $({\it l}_m^{\,0})_n$ denotes the length of $\g_n(\p,0)$ between $-\a(m)$ and $\a(m)$, then we define
$$\psi(m,0)_n:=\f{L_n}{\pi}\sin\lf\f{\pi ({\it l}_{m}^{\,0})_n}{L_n}\rg\quad\mbox{and}\quad\mathcal{G}(m,0)_n=\f{\psi(m,0)_n}{({\it d}_m^{\,0})_n}.$$

On the other hand, inspired by \cite{Distanceprincipio}, we define an isoperimetric ratio for the Grim Reaper $\g_\infty$ as follows:

Let $(\G_m)_\infty$ be the straight line that joins the points $\g_\infty(-\a(m),0)$ with $\g_\infty(\a(m),0).$ If $({\it d}_{m})_\infty$ denotes the length of $(\G_m)_\infty$ and $({\it l}_{m})_\infty$ denotes the length of $\g_\infty(\p,0)$ between $\g_\infty(-\a(m),0)$ and $\g_\infty(\a(m),0)$, then we define
$$\mathcal{G}(m)_\infty=\f{({\it l}_m)_\infty}{({\it d}_m)_\infty}.$$

Thus, we prove
\begin{Lemma}
\label{radioinf}
The isoperimetric ratio $\mathcal{G}(m)_\infty$ converges to infinity when $m$ tends to infinity.
\end{Lemma}

\begin{proof}
Since $\g_\infty(x,0)=(x,-\ln(\cos(x))),$ then it is easy to see that
$$\lim_{m\to\infty}({\it l}_m)_\infty=\infty\quad\mbox{and}\quad \lim_{m\to\infty}({\it d}_m)_\infty=\pi,$$

which finishes the proof.
\end{proof}

\begin{Lemma}
\label{boundedG}
Under the same hypotheses in Theorem A, the isoperimetric ratio for the rescaled solution $\g_n(\p,0)$ is bounded for every $n\in\N.$
\end{Lemma}

\begin{proof}
The proof is direct of Lemma \ref{evolutionradio2} and Remark \ref{invariante}.
\end{proof}

Additionally, if we assume that the evolution of $\g(\p,t)$ forms a Type II singularity at $\w$, Proposition \ref{TypeII} implies that given $\e>0$ there exists $N\in\N$ such that $\forall n\geq N$ we have that $\g_n|_{\mathcal{K}_m}$ converges uniformly in $C^1$ to $\g_\infty|_{\mathcal{K}_m}$, i.e. given $m$ fixed we have
\begin{equation}
\label{gnconvergeginf}
\lim_{n\to\infty}\g_n|_{\mathcal{K}_m}=\g_\infty|_{\mathcal{K}_m}\quad \mbox{and} \quad \lim_{n\to\infty}\g_n'|_{\mathcal{K}_m}=\g_\infty'|_{\mathcal{K}_m}.
\end{equation}

Now, to understand the convergence of the sequence of isoperimetric ratios $\{\mathcal{G}_n\}_{n\in\N},$ we need to analyze the convergence of the sequence of minimal surfaces $\{X_n(\p,0)\}_{n\in\N}$ under the assumption that the Gaussian curvature of $X(\p,t)$ is uniformly bounded in time.

Theorem \ref{convergeofsurface} and Remark \ref{convergenceuptheboundary} state that the sequence of minimal surfaces $\{X_n(\p,0)\}_{n\in\N}$ converges to a planar surface up to its boundary $\g_n(\p,0)$ implying that the piecewise geodesic on the surface $(\G_m)_n(\p,0)$ converges ($C^1$ at least) to a planar curve $(\tilde{\G}_m)_\infty$ that joins the points $\g_\infty(-\a(m),0)$ with $\g_\infty(\a(m),0)$.

Moreover, Proposition \ref{curvaturetendzero} states that the limit curve $(\tilde{\G_m})_\infty$ is the straight line that joins the points $\g_\infty(-\a(m),0)$ with $\g_\infty(\a(m),0)$.

Therefore, we get
\begin{Proposition}
\label{convergencegeodesic}
The piecewise geodesic $(\G_m)_n(\p,0)\subset X_n(\p,0)$ that minimizes distance between $\g_n(-\a(m),0)$ and $\g_n(\a(m),0)$ converges ($C^1$ at least) to the straight line $(\G_m)_\infty(\p,0)$ that joins $\g_\infty(-\a(m),0)$ with $\g_\infty(\a(m),0).$
\end{Proposition}

Thus, we prove
\begin{Lemma}
\label{radioconvergencia}
Under the same hypotheses of Theorem A, the isoperimetric ratio $\mathcal{G}(m,0)_n$ converges to $\mathcal{G}(m)_\infty$ as $n$ converges to infinity.
\end{Lemma}

\begin{proof}
Firstly, following \eqref{gnconvergeginf} we have
\begin{equation*}
\lim_{n\to\infty}\lf{\it l}^0_{m}\rg_n=\lim_{n\to\infty}\I{-\a(m)}{\a(m)}\|\g_n'\|dx=\I{-\a(m)}{\a(m)}\lim_{n\to\infty}\|\g_n'\|dx=\lf{\it l}_{m}\rg_\infty.
\end{equation*}

Moreover, note that $L_n$ tends to infinity when $n$ tends to infinity. Thus, we have
\begin{equation}
\label{limite psi}
\begin{split}
\lim_{n\to\infty}\psi(m,0)_n&=\lim_{n\to\infty}\f{L_n}{\pi}\sin\lf\f{\pi ({\it l}_{m}^{\,0})_n}{L_n}\rg\\
&=\lim_{n\to\infty}({\it l}_{m}^{\,0})_n\p\f{L_n}{\pi ({\it l}_{m}^{\,0})_n}\sin\lf\f{\pi ({\it l}_{m}^{\,0})_n}{L_n}\rg\\
&=({\it l}_{m})_\infty.
\end{split}
\end{equation}

On the other hand, using Proposition \ref{convergencegeodesic}, for fixed $m$ we have that $(\G_m)_n'$ uniformly converges to $(\G_m)_\infty'$, then
\begin{equation*}
\lim_{n\to\infty}\lf{\it d}^0_{m}\rg_n=\lim_{n\to\infty}\I{0}{1}\|(\G_m)_n'\|dx=\I{0}{1}\lim_{n\to\infty}\|(\G_m)_n'\|dx=\lf{\it d}_{m}\rg_\infty.
\end{equation*}

Thus, we obtain
$$\lim_{n\to\infty}\mathcal{G}(m,0)_n=\lim_{n\to\infty}\f{\psi(m,0)_n}{({\it d}_m^{\,0})_n}=\f{({\it l}_{m})_\infty}{\lf{\it d}_{m}\rg_\infty}=\mathcal{G}(m)_\infty.$$
\end{proof}

Therefore, we conclude
\begin{proof}[Proof of Theorem A]
Assume that $\g(\p,t)$ satisfies the hypothesis of Theorem A and as $t\to \w$ forms a singularity Type II.

Since for every $n$ and $m$ we have
\begin{equation}
\label{paso1caso2}
\left|\mathcal{G}(m)_\infty\right|\leq\left|\mathcal{G}(m)_\infty-\mathcal{G}(m,0)_n\right|+\left|\mathcal{G}(m,0)_n\right|,
\end{equation}

Lemma \ref{radioinf}, \ref{boundedG} and \ref{radioconvergencia} give us a contradiction for large enough $m$.
\end{proof}

\bigskip
\bigskip
\medskip
\begin{flushright}
{\it Departamento de Matem\'aticas, Universidad de Chile,\\
Las Palmeras 3425, Casilla 653, Santiago, Chile}\\
\emph{E-mail:} karence1988@gmail.com
\end{flushright}
\end{document}